\documentclass[12pt]{article}
\usepackage{amsmath}
\usepackage{amssymb}
\usepackage{amsthm}

\allowdisplaybreaks

\newcommand\Sec{\mathop{\mathrm{Sec}}\nolimits}
\newcommand\hh{\mathbf h}
\newcommand\HH{\mathbf H}
\newcommand\BB{\mathbf B}
\newcommand\vv{\mathbf v}
\newtheorem{lem}{Lemma}
\newtheorem*{thm}{Theorem}
\theoremstyle{remark}
\newtheorem*{rem}{Remark}

\makeatletter 
\def\fvecfill{$\m@th\hbox{\raisebox{-5.75pt}[1.5pt][0pt]{$\!\mathord\ulcorner$}}
\mkern-5mu
\cleaders\hbox{\raisebox{-0.5pt}[1.5pt][0pt]{$\!\mathord-$}}\hfill
\mkern-5mu
\mathord{\hbox{\raisebox{-5.75pt}[1.5pt][0pt]{$\!\mathord\urcorner\!$}}}$}
\makeatother
\def\fentkapcs { \overset{\hbox{\fvecfill}} } 

\begin{document}
\title{Homotheties of Finsler manifolds\thanks{Supported by National Science Research Foundation OTKA No.\ NK68040.}}
\author{Rezs\H o L.~Lovas\and J\'ozsef Szilasi}
\date{}
\maketitle
Keywords: Finsler manifold, homothety.

AMS classification: 53C60.

\begin{abstract}
We give a new and complete proof of the following theorem, discovered by Detlef Laugwitz: (forward) complete and connected finite dimensional Finsler manifolds admitting a proper homothety are Minkowski vector spaces. More precisely, we show that under these hypotheses the Finsler manifold is isometric to the tangent Minkowski vector space of the fixed point of the homothety via the exponential map of the canonical spray of the Finsler manifold.
\end{abstract}

\section{Introduction and history}

In the 17th century John Wallis managed to prove Euclid's parallel postulate
(EPP) by assuming a new axiom: \emph{to every figure there exists a similar
figure of arbitrary magnitude}. Later in his book `Euclid Vindicated from every
flaw' (1733) G.~G.~Saccheri pointed out that Wallis could have proved EPP by assuming only the existence of \emph{two} similar but uncongruent \emph{triangles}. Wallis' clever observation implies the collapse of similarity theory in hyperbolic geometry. After the discovery of Riemannian geometry it turned out that this phenomenon is almost typical: Riemannian manifolds admitting a proper similitude are rare. More precisely, Euclidean spaces are \emph{characterized} by the existence of a proper similitude among (complete and connected) Riemannian manifolds. Somewhat surprisingly it is not easy to find a complete proof of this important fact in the literature (at least we did not manage to find one). In his excellent textbook \emph{Differential and Riemannian Geometry} Detlef Laugwitz formulates the statement as follows:
\begin{quote}
`If a complete Riemannian space allows a proper similarity mapping onto itself,
then it is a Euclidean space.'
\end{quote}
(see \cite[Theorem 13.6.1]{Laug1}). We think, however, that his proof is incomplete. Laugwitz shows that if there is a contractive homothety of the manifold, then it has a fixed point, and the holonomy group is trivial in the fixed point. From this he immediately deduces that the space is Euclidean. This implication is in general false (think of the flat torus), and he does not explain why it is still true in this case. We note that Kobayashi and Nomizu state a weaker assertion in their book (indirectly): there exists a \emph{local} isometry from a neighbourhood of the point into a Euclidean space \cite[p.\ 242, Lemma 2]{KN}, see also \cite{Kob}.

The theorem was also generalized to Finsler manifolds by Laugwitz in the following form \cite[p.\ 268]{Laug}:
\begin{quote}
`A complete connected finite dimensional Finsler space $M$ which admits a
similitude $S$ is a Minkowski space.'
\end{quote}
The proof of this theorem (which implies immediately the Riemannian version)
seems to us also rather incomplete. Laugwitz shows that the Finsler manifold is
flat in the case of the existence of a contractive homothety. Then he refers to
p.\ 136 of Rund's monograph \cite{Rund}, and applies \'E.~Cartan's
characterization of `Minkowskian spaces'. However, the terminology of Rund's
book is strongly misleading here: `Minkowskian' actually means `locally
Minkowskian', i.e., a Finsler manifold with the property that all of its points
have a neighbourhood over which the Finsler function `depends only on the
position' (see \cite[3.2.4]{Matsumoto} and \cite[3.14, Theorem 3]{Setting}).

In common with E.~Heil, Laugwitz published another proof of the theorem \cite{HL}. This is much more convincing, but also suffers from some weakness: the use of the holonomy group of Rund's connection needs a much more careful elaboration, and the blowing up argument in the last step of the proof is far from being rigorous.

In this paper we wish to provide a new and self-contained proof of  Laugwitz's nice theorem which is already free from the flaws made by him (and them). Therefore we start along partially new lines. We use the simplest covariant derivative operator in Finsler geometry, Berwald's covariant derivative, rather than the Chern\,--\,Rund derivative (called only Rund derivative by Laugwitz). We avoid the use of the holonomy group, since we think that its complicatedness obscures the main points rather than clarifies them; instead, we only use the curvature tensors of Berwald's derivative. The use of Banach's fixed point theorem on a contractive homothety is rather standard, thus the first part of our proof, where we show that the manifold is locally Minkowskian, largely follows the proof of the Riemannian analogue in \cite{KN}. The essential new point in the proof is that a suitable global isometry is provided by the exponential map of a covariant derivative operator on the \emph{base manifold} at the fixed point of the homothety. The proof of the fact that this exponential map is a local isometry was inspired by \cite{Wilk}. The main original idea is the proof that it is a global isometry as well in our case.

\section{Notation and setup}

The term `manifold' will always mean a finite dimensional, connected smooth manifold which is Hausdorff and has a countable basis of open sets. If $M$ is a manifold, then $C^\infty(M)$ is the ring of smooth real-valued functions on $M$, $\tau:TM\to M$ is the tangent bundle of $M$, and $\mathfrak X(M)$ denotes the the $C^\infty(M)$-module of (smooth) vector fields of $M$. The tangent map of a smooth map $\varphi$ will be denoted by $\varphi_*$. If $D$ is a covariant derivative operator on $M$, $c$ is a smooth curve in $M$, and $\mathfrak X(c)$ is the module of smooth vector fields along $c$, then the induced covariant derivative operator along $c$ will be denoted by $D_c$. Let $R$ be the curvature tensor of $D$, and suppose that $c$ is a geodesic of $D$. We recall that a vector field $J\in\mathfrak X(c)$ is said to be a \emph{Jacobi field} if it satisfies the Jacobi equation $D_cD_cJ=R(\dot c,J)\dot c$. This is a second-order linear differential equation on the components of $J$, thus, given any vectors $v,w\in T_{c(0)}M$, there exists a unique Jacobi field $J$ defined on the whole domain of $c$ such that $J(0)=v$, $D_cJ(0)=w$.

Let $\overset\circ TM:=TM\setminus o(M)$, where $o\in\mathfrak X(M)$ is the zero
vector field, and consider the vector bundle $\overset\circ\pi:\overset\circ
TM\times_M TM\to\overset\circ TM$. Its fibre over $v$ is canonically isomorphic to the tangent space $T_{\tau(v)}M$, and hence the $C^\infty\big(\overset\circ TM\big)$-module $\Sec\big(\overset\circ\pi\big)$ of the sections of $\overset\circ\pi$ may be viewed as
\[
\Big\{\tilde X:\overset\circ TM\to TM\Big|
\tilde X\mbox{ is smooth, and }\tau\circ\tilde X=\tau\upharpoonright\overset\circ TM\Big\}.
\]
The module $\Sec\big(\overset\circ\pi\big)$ is generated by the \emph{basic sections} $\hat X:=X\circ\tau$, $X\in\mathfrak X(M)$. We have a canonical $C^\infty\big(\overset\circ TM\big)$-linear injection
\[
\mathbf i:\Sec\big(\overset\circ\pi\big)\to\mathfrak X\big(\overset\circ TM\big),
\ \hat X\mapsto\mathbf i\hat X:=X^{\mathsf v}:=\tau\mbox{-vertical lift of }X,
\]
and a canonical $C^\infty\big(\overset\circ TM\big)$-linear surjection $\mathbf j:\mathfrak X\big(\overset\circ TM\big)\to\Sec\big(\overset\circ\pi\big)$ such that for all $X\in\mathfrak X(M)$,
\[
\mathbf jX^{\mathsf v}=0,\ \mathbf jX^{\mathsf c}=\hat X
\]
($X^{\mathsf c}$ is the complete lift of $X$). Since $\mathbf i$ and $\mathbf j$ are tensorial, they have a natural pointwise interpretation which will be used automatically.

The push-forwards of a vector field $X$ on $M$ and a section $\tilde X\in\Sec\big(\overset\circ\pi\big)$ by a diffeomorphism $\varphi:M\to M$ are
\[
\varphi_\sharp X:=\varphi_*\circ X\circ\varphi^{-1}\mbox{ and }
\varphi_\sharp\tilde X:=(\varphi_*\times\varphi_*)\circ\tilde X\circ\varphi^{-1},\mbox{ resp}.
\]
By an \emph{Ehresmann connection} over $M$ we mean a $C^\infty\big(\overset\circ
TM\big)$-linear map $\mathcal H$ from $\Sec\big(\overset\circ\pi\big)$ into
$\mathfrak X\big(\overset\circ TM\big)$ such that $\mathbf j\circ\mathcal H$ is
the identity of $\Sec\big(\overset\circ\pi\big)$. To an Ehresmann connection
$\mathcal H$ we associate the horizontal projector $\mathbf h:=\mathcal
H\circ\mathbf j$, the vertical projector $\mathbf v:=1-\mathbf h$ and the vertical map $\mathcal V:=\mathbf i^{-1}\circ\mathbf v$. (For details we refer to \cite{PSz,Setting,SzL}.)

If $S$ is a spray over $M$ in the sense of Dazord \cite{Daz} (see also \cite{Grifone,Setting}), then a diffeomorphism $\varphi:M\to M$ is said to be an \emph{automorphism of $S$} if \mbox{$(\varphi_*)_\sharp S=S$.} An \emph{affinity} of $S$ is a diffeomorphism $\varphi$ of $M$ such that $\varphi\circ c$ is a geodesic of $S$ whenever $c$ is a geodesic of $S$. If $\mathcal H$ is an Ehresmann connection over $M$, then an \emph{automorphism of $\mathcal H$} is a diffeomorphism $\varphi:M\to M$ such that $\varphi_{**}\circ\mathcal H=\mathcal H\circ(\varphi_*\times\varphi_*)$. Finally, if
\[
\nabla:\mathfrak X\big(\overset\circ TM\big)\times\Sec\big(\overset\circ\pi\big)
\to\Sec\big(\overset\circ\pi\big)
\]
is a covariant derivative operator along $\overset\circ\pi$, then a diffeomorphism $\varphi$ of $M$ is called an \emph{automorphism of $\nabla$} if
\[
\varphi_\sharp\nabla_\xi\tilde Y=\nabla_{(\varphi_*)_\sharp\xi}\varphi_\sharp\tilde Y
\quad\Big(\xi\in\mathfrak X\big(\overset\circ TM\big),
\tilde Y\in\Sec\big(\overset\circ\pi\big)\Big).
\]

\section{Finsler manifolds}

A continuous function $F:TM\to\left[0,\infty\right[$ is said to be a
\emph{Finsler function} over a manifold $M$ if it is smooth on $\overset\circ
TM$, positive-homogeneous of degree 1, i.e., $F(\lambda v)=\lambda F(v)$ for all
$\lambda\in\left[0,\infty\right[$ and $v\in TM$, and has the property that the \emph{metric tensor} $g:\Sec\big(\overset\circ\pi\big)\times\Sec\big(\overset\circ\pi\big)\to C^\infty\big(\overset\circ TM\big)$ defined on the basic vector fields by
\[
g\big(\hat X,\hat Y\big):=\frac 12X^{\mathsf v}Y^{\mathsf v}F^2;\ X,Y\in\mathfrak X(M)
\]
is pointwise nondegenerate. Then $g$ is obviously symmetric, and it can also be shown that our conditions on $F$ imply the positive definiteness of the metric tensor \cite L. A manifold equipped with a Finsler function is said to be a \emph{Finsler manifold}. By the \emph{length} of a piecewise smooth curve $\gamma:[\alpha,\beta]\to M$ in a Finsler manifold $(M,F)$ we mean the integral $L(\gamma):=\int_\alpha^\beta F\circ\dot\gamma$. If $p$ and $q$ are points of $M$, and
\[
\Gamma(p,q):=\{\gamma:[\alpha,\beta]\to M|\gamma\mbox{ is piecewise smooth, and } \gamma(\alpha)=p,\ \gamma(\beta)=q\},
\]
then the function
\[
\varrho_F:M\times M\to\mathbb R,\ (p,q)\mapsto\varrho_F(p,q)
:=\inf_{\gamma\in\Gamma(p,q)}L(\gamma)
\]
is a \emph{quasi-distance} on $M$, i.e., has the following properties:
\begin{enumerate}
\item
$\varrho_F(p,q)\ge 0$, $\varrho_F(p,q)=0$, if and only if, $p=q$;
\item
$\varrho_F(p,s)\le\varrho_F(p,q)+\varrho_F(q,s)$ for all $p,q,s\in M$;
\item
the \emph{forward metric balls} $B_r^+(a):=\{p\in M|\varrho_F(a,p)<r\}$ and the \emph{backward metric balls} $B_r^-(a):=\{p\in M|\varrho_F(p,a)<r\}$ generate the same topology (namely, the underlying manifold topology)
\end{enumerate}
\cite{BCS,Pat}. Thus $(M,F)$ becomes a \emph{quasi-metric space} $(M,\varrho_F)$.

By the property (iii), there is a well-defined notion of the convergence of a
sequence in our quasi-metric space $(M,\varrho_F)$, thus there is no need to
speak about `forward convergence' and `backward convergence'. On the other hand, we have to distinguish between the notions of a forward Cauchy sequence and a backward Cauchy sequence. Namely, a sequence $(p_n)_{n\in\mathbb N}$ in $M$ is said to be a \emph{forward Cauchy sequence} if for any positive $\varepsilon$ there is a number $N\in\mathbb N$ such that
\[
\varrho_F(p_m,p_n)<\varepsilon\quad\mbox{whenever }N\le m\le n,
\]
and $(p_n)_{n\in\mathbb N}$ is said to be a \emph{backward Cauchy sequence} if for any positive $\varepsilon$ there is a number $N\in\mathbb N$ such that
\[
\varrho_F(p_n,p_m)<\varepsilon\quad\mbox{whenever }N\le m\le n.
\]
The quasi-metric space $(M,\varrho_F)$ is \emph{forward (backward) complete} if every forward (backward) Cauchy sequence converges, respectively. From the proof of Banach's fixed point theorem it is easy to see that it remains true for quasi-metric spaces if either of the two completeness properties is required. However, in our main theorem we shall assume forward rather than backward completeness, since forward completeness implies that the exponential map of the Finsler manifold $(M,F)$ is defined on the whole of each tangent space of $M$ \cite{BCS}.

A diffeomorphism $\varphi$ of $M$ is said to be a \emph{homothety} of the Finsler manifold $(M,F)$ if there is a positive real number $\lambda$ such that $F\circ\varphi_*=\lambda F$. If $\lambda=1$, then $\varphi$ is called an \emph{isometry} of $(M,F)$. A homothety is \emph{proper} if it is not an isometry. \emph{A homothety $\varphi$ of $(M,F)$ with proportionality factor $\lambda$ is also a homothety of the quasi-metric space $(M,\varrho_F)$ with the same proportionality factor.} Indeed, for any points $p,q\in M$ we obtain
\begin{align*}
\varrho_F(\varphi(p),\varphi(q))
&=\inf_{\gamma\in\Gamma(p,q)}\int_\alpha^\beta F\circ\dot{\fentkapcs{\varphi\circ\gamma}}
=\inf_{\gamma\in\Gamma(p,q)}\int_\alpha^\beta F\circ\varphi_*\circ\dot\gamma\\
&=\inf_{\gamma\in\Gamma(p,q)}\int_\alpha^\beta\lambda F\circ\dot\gamma
=\lambda\inf_{\gamma\in\Gamma(p,q)}\int_\alpha^\beta F\circ\dot\gamma=\lambda\varrho_F(p,q).
\end{align*}

\begin{rem}
In a rather forgotten paper \cite{Brick} F.~Brickell showed that a homeomorphism of a manifold equipped with a spray onto itself is a diffeomorphism if it preserves the geodesics considered as parametrized curves. Applyig this result he deduced that the isometry group of the quasi-metric space $(M,\varrho_F)$ coincides with the isometry group of the Finsler manifold $(M,F)$, generalizing a well-known theorem of S.~B.~Myers and N.~E.~Steenrod from Riemannian geometry. (This result of Brickell was rediscovered by S.~Deng and Z.~Hou \cite{DH}.) Then it follows that the isometry group of $(M,\varrho_F)$ is a Lie group, which implies, as M.~Patr\~ao showed \cite{Pat}, that \emph{the homothety group of $(M,\varrho_F)$ is also a Lie group}.
\end{rem}

\section{The Berwald connection. Curvatures}

First we recall that the Liouville vector field on $TM$ is the velocity field of the flow $(t,v)\in\mathbb R\times TM\mapsto e^tv\in TM$; it will be denoted by $C$.

We now come to what should be considered as the \emph{`fundamental lemma of
Finsler geometry'}. If $(M,F)$ is a Finsler manifold, then there exists a unique Ehresmann connection $\mathcal H$ over $M$ such that
\begin{enumerate}
\item
$\big[\mathcal H\hat X,C\big]=0$ for all $X\in\mathfrak X(M)$ ($\mathcal H$ is \emph{homogeneous}),
\item
$\big[\mathcal H\hat X,Y^{\mathsf v}\big]-\big[\mathcal H\hat Y,X^{\mathsf v}\big]-[X,Y]^{\mathsf v}=0$ for all $X,Y\in\mathfrak X(M)$ ($\mathcal H$ is \emph{torsion-free}),
\item
$dF\circ\mathcal H=0$ ($\mathcal H$ is \emph{conservative}).
\end{enumerate}

This connection is said to be the \emph{Berwald connection} of $(M,F)$. For a proof we refer to \cite{Grifone,Setting} (see also \cite{SzL}); we only note that if $S$ is the \emph{canonical spray} of the Finsler manifold determined by the Euler\,--\,Lagrange equation
\[
i_S\omega=-dF^2,\quad\omega(X^{\mathsf v},Y^{\mathsf c})
:=g\big(\hat X,\hat Y\big)
\quad(X,Y\in\mathfrak X(M)),
\]
then we have
\[
\mathcal H\big(\hat X\big)=\frac 12(X^{\mathsf c}+[X^{\mathsf v},S]),\quad X\in\mathfrak X(M).
\]
Using Berwald's connection, we define a Riemannian metric $\bar g$ on $\overset\circ TM$, nicely related to the metric tensor $g$ of $(M,F)$, as follows:
\[
\bar g(\xi,\eta):=g(\mathbf j\xi,\mathbf j\eta)+g(\mathcal V\xi,\mathcal V\eta);
\ \xi,\eta\in\mathfrak X\big(\overset\circ TM\big).
\]
Berwald's connection determines a covariant derivative operator
\[
\nabla:\mathfrak X\big(\overset\circ TM\big)\times\Sec\big(\overset\circ\pi\big)
\to\Sec\big(\overset\circ\pi\big)
\]
in $\overset\circ\pi$ by the rule
\[
\nabla_\xi\tilde Y:=\mathbf j\big[\mathbf v\xi,\mathcal H\tilde Y\big]
+\mathcal V\big[\mathbf h\xi,\mathbf i\tilde Y\big];
\ \xi\in\big(\overset\circ TM\big),\ \tilde Y\in\Sec\big(\overset\circ\pi\big),
\]
called \emph{Berwald's derivative}. Let $R^\nabla$ be the classical curvature tensor of $\nabla$. Then the type $(1,3)$ tensors $\mathbf H$ and $\mathbf B$ over $\Sec\big(\overset\circ\pi\big)$ given by
\[
\mathbf H\left(\tilde X,\tilde Y\right)\tilde Z
:=R^\nabla\left(\mathcal H\tilde X,\mathcal H\tilde Y\right)\tilde Z
\mbox{ and }\mathbf B\left(\tilde X,\tilde Y\right)\tilde Z
:=R^\nabla\left(\mathbf i\tilde X,\mathcal H\tilde Y\right)\tilde Z
\]
are said to be the \emph{affine} and the \emph{Berwald curvature} of $(M,F)$, respectively. An immediate calculation shows that for any sections $\tilde X,\tilde Y,\tilde Z$ in $\Sec\big(\overset\circ\pi\big)$ we have
\[
R^\nabla\left(\mathbf i\tilde X,\mathbf i\tilde Y\right)\tilde Z=0,
\]
therefore \emph{$R^\nabla$ vanishes, if and only if, the affine and the Berwald curvature of $(M,F)$ vanish}. In this case we say that the Finsler manifold $(M,F)$ is \emph{flat}. For some equivalents of flatness we refer to \cite[3.14, Theorem 3]{Setting}. Notice that flat Finsler manifolds are usually mentioned as \emph{locally Minkowski spaces}.

\section{The main result}

After two preparatory lemmas we prove the main result of the paper.

\begin{lem}
If $\varphi:M\to M$ is a homothety of a Finsler manifold $(M,F)$, then it is an automorphism of Berwald's covariant derivative.
\end{lem}
\begin{proof}
By the local length minimizing property of geodesics of a Finsler manifold, $\varphi$ is an affinity of the canonical spray $S$. Then, by \cite[Lemma 5.1]{PSz}, $\varphi$ is also an automorphism of $S$, and by \cite[Lemma 6.1]{PSz}, it is thus an automorphism of Berwald's connection $\mathcal H$ as well. Finally, our assertion follows from \cite[Lemma 7.2]{PSz}.
\end{proof}

\begin{lem}
Let $D$ be a covariant derivative on $M$, $p\in M$, and
\[
\exp_p:U\subset T_pM\to M
\]
be the exponential map of $M$ at $p\in M$.
Let $v\in U$, and let $c$ (defined on an open interval containing $0$ and $1$) be the geodesic such that $c(0)=p$ and $\dot
c(0)=v$. Let $w\in T_pM$, and let $J$ be the Jacobi field
such that $J(0)=0$ and $D_cJ(0)=w$. Then we have
\[
\left(\exp_p\right)_*(w_v)=J(1).
\]
\end{lem}

This is essentially a reformulation of \cite[Chap.\ 5, 2.5 Corollary]{doCar}, see also \cite[Chapter IX, Theorem 3.1]{Lang}.

\begin{thm}
If a forward complete connected finite-dimensional Finsler manifold admits a proper
homothety onto itself, then it is isometric to a Minkowski vector space, namely, to the tangent Minkowski vector space at the fixed point of the homothety.
\end{thm}
\begin{proof}
Let $(M,F)$ be our Finsler manifold and $\varphi$ be a homothety of $(M,F)$ with proportionality factor $\lambda$. We may
assume that $0<\lambda<1$ (otherwise take $\varphi^{-1}$ instead of $\varphi$). Then, as we have just seen, $\varphi$ is also a homothety of $(M,\varrho_F)$, thus,
Banach's fixed point theorem implies the existence of a unique fixed point $p$ of $\varphi$,
i.e., a point $p\in M$ such that $\varphi(p)=p$.

First we prove that $(M,F)$ is flat. In our calculations we shall use the fact
that the curvatures $R^\nabla$, $\HH$ and $\BB$ are tensorial in all of their arguments, thus, they can also be evaluated on single vectors rather than vector fields on $\overset\circ TM$ and along $\overset\circ\tau$.

Let $U$ be an open neighbourhood of $p$ such that $\overline U$ is compact, and let
\begin{multline*}
r:=\sup\left\{g_u\left(R^\nabla_u(z_1,z_2)v,w\right)\big|
u\in TU,z_1,z_2\in T_uTM,v,w\in T_{\tau(u)}M,\right.\\
F(u)=1,\bar g_u(z_1,z_1)=\bar g_u(z_2,z_2)=g_u(v,v)=g_u(w,w)=1\big\}.
\end{multline*}
Now we show that
\begin{equation*}
R^\nabla_{\varphi_*u}(\varphi_{**}z_1,\varphi_{**}z_2)\varphi_*v
=\varphi_*R^\nabla_u(z_1,z_2)v\tag{$\ast$}
\end{equation*}
for any $q\in M$, $u\in\overset\circ T_qM$, $z_1,z_2\in T_uTM$, $v\in T_qM$.
Indeed, let $\xi,\eta$ be vector fields on $\overset\circ TM$ such that
$\xi(u)=z_1$, $\eta(u)=z_2$, and $Z\in\mathfrak X(M)$ such that $Z(q)=v$. Since,
by Lemma 1, $\varphi$ is an automorphism of $\nabla$, $R^\nabla$ is preserved by $\varphi_*$. Thus we obtain
\begin{multline*}
R^\nabla_{\varphi_*u}(\varphi_{**}z_1,\varphi_{**}z_2)\varphi_*v
=R^\nabla_{\varphi_*u}(\varphi_{**}\xi(u),\varphi_{**}\eta(u))\varphi_*Z(u)\\
=\left(R^\nabla((\varphi_*)_\sharp\xi,(\varphi_*)_\sharp\eta)\widehat{\varphi_\sharp Z}\right)(\varphi_*u)
=\varphi_\sharp\left(R^\nabla(\xi,\eta)\hat Z\right)(\varphi_*u)\\
=\varphi_*\left(R^\nabla_u(\xi(u),\eta(u))Z(u)\right)=\varphi_*R^\nabla_u(z_1,z_2)v,
\end{multline*}
which proves ($\ast$). If, in addition, $w\in T_qM$, then we have
\begin{multline*}
g_{\varphi_*u}\left(R^\nabla_{\varphi_*u}(\varphi_{**}z_1,\varphi_{**}z_2)\varphi_*v,
\varphi_*w\right)\\
=g_{\varphi_*u}\left(\varphi_*R^\nabla_u(z_1,z_2)v,\varphi_*w\right)
=\lambda^2g_u\left(R^\nabla_u(z_1,z_2)v,w\right),
\end{multline*}
and, if $n\in\mathbb N$, then
\[
g_{\varphi^n_*u}\left(R^\nabla_{\varphi^n_*u}(\varphi^n_{**}z_1,\varphi^n_{**}z_2)\varphi^n_*v,
\varphi^n_*w\right)=\lambda^{2n}g_u\left(R^\nabla_u(z_1,z_2)v,w\right)
\]
by induction. Now suppose that
\[
F(u)=1\quad\mbox{and}\quad\bar g_u(z_1,z_1)=\bar g_u(z_2,z_2)=g_u(v,v)=g_u(w,w)=1.
\]
In that case we have
\begin{multline*}
F(\varphi^n_*u)=\lambda^n\quad
\mbox{and}\quad\bar g_{\varphi^n_*(u)}(\varphi^n_{**}z_1,\varphi^n_{**}z_1)
=\bar g_{\varphi^n_*u}(\varphi^n_{**}z_2,\varphi^n_{**}z_2)\\
=g_{\varphi^n_*u}(\varphi^n_{**}v,\varphi^n_{**}v)
=g_{\varphi^n_*u}(\varphi^n_*w,\varphi^n_*w)=\lambda^{2n}.
\end{multline*}
The sequence $(\varphi^n(q))_{n\in\mathbb N}$ converges to $p$, thus there is an index $n_0\in\mathbb N$ such that $\varphi^n(q)\in U$ for all $n\ge n_0$. Therefore
\begin{align*}
&\left|g_{\varphi^n_*u}\left(R^\nabla_{\varphi^n_*u}(\varphi^n_{**}z_1,\varphi^n_{**}z_2)
\varphi^n_*v,\varphi^n_*w\right)\right|\\
&=\left|g_{\varphi^n_*u}\left(R^\nabla_{\varphi^n_*u}
(\hh\varphi^n_{**}z_1,\hh\varphi^n_{**}z_2)
\varphi^n_*v,\varphi^n_*w\right)\right.\\
&\quad+g_{\varphi^n_*u}\left(R^\nabla_{\varphi^n_*u}
(\vv\varphi^n_{**}z_1,\hh\varphi^n_{**}z_2)
\varphi^n_*v,\varphi^n_*w\right)\\
&\quad\left.+g_{\varphi^n_*u}\left(R^\nabla_{\varphi^n_*u}
(\hh\varphi^n_{**}z_1,\vv\varphi^n_{**}z_2)
\varphi^n_*v,\varphi^n_*w\right)\right|\\
&\le\left|g_{\varphi^n_*u}\left(\HH_{\varphi^n_*u}(\mathbf j\varphi^n_{**}z_1,
\mathbf j\varphi^n_{**}z_2)\varphi^n_*v,\varphi^n_*w\right)\right|\\
&\quad+\left|g_{\varphi^n_*u}\left(\BB_{\varphi^n_*u}(\mathcal V\varphi^n_{**}z_1,
\mathbf j\varphi^n_{**}z_2)\varphi^n_*v,\varphi^n_*w\right)\right|\\
&\quad+\left|g_{\varphi^n_*u}\left(\BB_{\varphi^n_*u}(\mathcal V\varphi^n_{**}z_2,
\mathbf j\varphi^n_{**}z_1)\varphi^n_*v,\varphi^n_*w\right)\right|\\
&\le\left(\lambda^{4n}+2\lambda^{3n}\right)r
\end{align*}
if $n\ge n_0$. Comparing the two expressions for the curvature tensor, we obtain
\[
\left|g_u\left(R^\nabla_u(z_1,z_2)v,w\right)\right|
\le\left(\lambda^{2n}+2\lambda^n\right)r\quad(n\ge n_0),
\]
from which we have $g_u\left(R^\nabla_u(z_1,z_2)v,w\right)=0$ by taking the limit $n\to\infty$. Non-degeneracy of $g$ implies that $R^\nabla_u(z_1,z_2)v=0$, and by the
homogeneity of $R^\nabla$ in $u$ and its linearity in its arguments, it follows that
$R^\nabla$ vanishes identically. This proves that $(M,F)$ is flat.

From the vanishing of $\BB$ it also follows that $\nabla$ is \emph{basic} in the
sense that there is a covariant derivative operator $D$ on $M$ such that
$\nabla_{\mathcal H\hat X}\hat Y=\widehat{D_XY}$ for any $X,Y\in\mathfrak X(M)$. Now we show
that $\exp_p:T_pM\to M$, the exponential map of $D$ in $p$ (which is just the exponential map associated to the canonical spray $S$), is a local
isometry between the vector space $T_pM$ equipped with the norm $F\upharpoonright T_pM$ and
the Finsler manifold $(M,F)$. (Due to the forward completeness of $M$, $\exp_p$ is indeed
defined on the whole of $T_pM$.) Let $v,w\in T_pM$ and $c:\left[0,\infty\right[\to M$
be as in Lemma 2, and let $X$ be the unique parallel vector field
along $c$ with $X(0)=w$. Since $\HH=0$, the
curvature of $D$ also vanishes, thus in this case the Jacobi equation has the very
simple form
\[
D_cD_cJ=0.
\]
If $J(t):=tX(t)$ ($t\in\left[0,\infty\right[$), then
\[
D_cJ(t)=X(t)+tD_cX(t)=X(t),\ D_cD_cJ=0,\mbox{ and }D_cJ(0)=X(0)=w,
\]
thus $J$ is just the Jacobi field along $c$ which features in Lemma 2. Since $X$
is parallel as a vector field along $c$, it is horizontal as a curve running in
$\overset\circ TM$. (We can ignore the trivial case when $w=0$.) Therefore $F$
is constant along $X$, and
\[
F\left(\left(\exp_p\right)_*(w_v)\right)=F(J(1))=F(X(1))=F(X(0))=F(w),
\]
which means that $\exp_p$ is a local isometry.

Finally we show that $\exp_p:T_pM\to M$ is in fact a (global) isometry. It is enough to
check that $\exp_p$ is injective, since its surjectivity will then follow from
the connectedness of $M$. Suppose, indirectly, that there are two
vectors $v,w\in T_pM$ with $\exp_p(v)=\exp_p(w)$, and consider the
parametrized straight line segment $t\in[0,1]\mapsto v+t(w-v)$, which is a
geodesic segment of the Minkowski vector space $T_pM$. Being a geodesic is a local
property, thus
\[
c:[0,1]\to M,\quad c(t):=\exp_p(v+t(w-v))\quad(t\in[0,1])
\]
is also a geodesic segment of $M$, whose starting point and end point coincide in addition. Let $\mathcal U\subset T_pM$ be
an open star-shaped neighbourhood of 0 such that $\exp_p\upharpoonright\mathcal U:\mathcal U\to\mathcal V\subset M$ is an isometry. The sets
\[
\left(\varphi^{-1}\right)^n(\mathcal V),\ n\in\mathbb N
\]
cover $M$, and since $c([0,1])$ is compact, there is an index $n_0$ such that
$c([0,1])\subset\left(\varphi^{-1}\right)^{n_0}(\mathcal V)$, or, equivalently,
$\varphi^{n_0}(c([0,1]))\subset\mathcal V$, thus $\varphi^{n_0}\circ c$ is a
geodesic segment in $\mathcal V$ with coinciding starting and end point, which contradicts the fact that $\mathcal V$ is isometric to an open
subset of a Minkowski space, in which geodesics are straight lines. This
completes the proof that $\exp_p$ is an isometry.
\end{proof}

\section*{Acknowledgement}

We are grateful to Professor Shoshichi Kobayashi for the stimulating discussion on the Riemannian aspects of the problem via e-mail.

\noindent
Rezs\H{o} L.~Lovas \textsf{lovasr@math.unideb.hu}\\
J\'ozsef Szilasi \textsf{szilasi@math.unideb.hu}\\
Institute of Mathematics, University of Debrecen\\
H--4010 Debrecen, P.O.Box 12, Hungary
\end{document}